%% file: main_primaldual.tex
\documentclass{article} 

\usepackage{hyperref}
\usepackage{enumerate}
\usepackage{xcolor}
\usepackage{geometry}
\usepackage{amsfonts,amsmath,bbm,amsthm}
\numberwithin{equation}{section}
\newcounter{balance}

\newcommand{\VC}[1]{\text{VC}(#1)}
\let\vec\mathbf

\input{macros}

\title{Primal and Dual Combinatorial Dimensions}

\author{
Pieter Kleer\\
Tilburg University\\
Tilburg, The Netherlands\\
\texttt{p.s.kleer@tilburguniversity.edu}
\and
Hans Simon\\ 
Max Planck Institute for Informatics\\
Saarland Informatics Campus (SIC)\\
Saarbr\"ucken, Germany\\
\texttt{hsimon@mpi-inf.mpg.de}
}

\begin{document}

\maketitle

\begin{abstract}%
\noindent We give tight bounds on the relation between the primal and dual  of various combinatorial dimensions, such as the pseudo-dimension and fat-shattering dimension, for multi-valued function classes. These dimensional notions play an important role in the area of learning theory. We first review some (folklore) results that bound the dual dimension of a function class in terms of its primal, and after that give (almost) matching lower bounds. In particular, we give an appropriate generalization to multi-valued function classes of a well-known bound due to Assouad (1983), that relates the primal and dual VC-dimension of a binary function class.
\end{abstract}

\section{Introduction}\label{sec:intro}
The Vapnik-Chervonenkis (VC) dimension \cite{vapnik1971uniform} is a fundamental combinatorial dimension in learning theory used to characterize the complexity of \emph{learning} a class $X$ consisting of functions $f: Y \rightarrow \{0,1\}$ where $X$ and $Y$ are given (possibly infinite) sets. Informally, the VC-dimension captures how rich or complex a class of functions is. Many extensions of the VC-dimension to multi-valued functions $f: X \rightarrow Z$, for some given $Z \subseteq \reals$, have been proposed in the literature, such as the Vapnik-dimension (also known as the uniform pseudo-dimension) \cite{Vapnik1989}, the Pollard-dimension (also known as pseudo-dimension) \cite{Pollard1990,Haussler1992}, and the fat-shattering dimension \cite{Kearns1994}. All these combinatorial dimensions are formally defined in Section \ref{sec:definitions}.

Every (primal) class of functions can be identified with a \emph{dual class} whose functions are of the form $g_y : X \ra Z$ for $y \in Y$ defined by $g_y(f) = f(y)$ for $f \in X$. When interpreting a function class as a matrix $A$ whose rows and columns are indexed by $X$ and $Y$, respectively, the dual class is simply given by the transpose matrix $A^\top$. The (VC, pseudo-, etc..) dimension  of the dual class is defined as the dimension of the matrix $A^T$. 

Assouad \cite{Assouad1983} showed the following relation between the primal VC-dimension $\VCD(A)$ and the dual VC-dimension $\VCD^*(A)$:
\begin{equation}
\VCD^*(A) \leq 2^{\VCD(A) + 1}-1.
\label{eq:assouad_intro}
\end{equation}
This has turned out to be a very useful inequality, e.g.,  in the context of so-called sample compression schemes \cite{Moran2016}. In case $\VCD^*(A)$ is a power of two, this immediately yields $\VCD^*(A) \leq 2^{\VCD(A)}$. It is known that this bound is tight for all values of $\VCD^*(A)$, see, e.g., \cite{Matouvsek2002}.\\

\noindent The purpose of this work is to understand the relation between the primal and dual of combinatorial dimensions for multi-valued function classes, in particular, for multi-valued functions where $Z = \{0,1,\dots,k\}$ for $k \in \mathbb{N}$. For the pseudo-dimension, as explained in Section \ref{sec:known}, it can be shown that
\begin{equation}
\PD^*(A) \leq k \cdot \left(2^{\PD(A) + 1}-1\right),
\end{equation}
which naturally generalizes Assouad's bound in \eqref{eq:assouad_intro}.\footnote{We refer to this as a folklore result, rather than a contribution of this work.} Again, when $\PD^*(A)$ is a power of two, this yields  
\begin{equation}
\PD^*(A) \leq k \cdot 2^{\PD(A)}.
\label{eq:dualpdim_intro}
\end{equation}
Our first contribution is that the bound in \eqref{eq:dualpdim_intro} is in fact tight for every value of $k$ and $\PD(A)$ (Theorem \ref{th:dualpdim-lb}). In case $\PD(A) = 1$, we give an improved bound of $k + 2$ (Theorem \ref{th:dualpdim-ub}), and also show that this is tight (Theorem \ref{th:dualpdim-lb}).
We obtain similar bounds for the fat-shattering dimension (Theorem \ref{th:dualfat-lb}).
\begin{remark}
It is sometimes believed that Assouad's bound also holds for combinatorial dimensions beyond the VC-dimension, see, e.g., \cite{hanneke2019sample}. Our results show that this is, unfortunately, not correct.
\end{remark}

\noindent \textit{Outline.} 
We continue in Section \ref{sec:definitions} with all the necessary definitions and notations, in particular the formal definitions of all combinatorial dimensions considered in this work. Then, in Section \ref{sec:known}, we outline known results regarding the relations between various combinatorial dimensions and their duals. After that, in Section \ref{sec:results}, we summarize our results, followed by their proofs in Section \ref{sec:proofs}.

\section{Preliminaries} \label{sec:definitions}

For $k\ge1$, we set $[k] := \{1,\ldots,k\}$ and $[k]_0 := [k] \cup \{0\}$.
Let $X$ and $Y$ be disjoint sets and let $Z \seq \reals$ be a subset of the reals. 
Consider a function  $A: X \times Y \ra Z$. For $x \in X$, 
we define $A_x: Y \ra Z$ by $A_x(y) = A(x,y)$ and refer to $A_x$ as a \emph{row} 
of $A$. For $y \in Y$, we define $A_y: X \ra Z$ by $A_y(x) = A(x,y)$ and refer 
to $A_y$ as a \emph{column} of $A$. The \emph{transpose} of $A$ is defined
as the function $A^\top: Y \times X \ra Z$ given by $A^\top(y,x) = A(x,y)$.
As suggested by this terminolgy, we view $A$ as a (possibly infinite) matrix 
with rows indexed by $X$, columns indexed by $Y$ and with $A^\top$ as its
transpose. 

A matrix $A: X \times Y \ra Z$ with $Z = \{0,1\}$ is said to be \emph{Boolean}.
Let $d \ge 1$ be a positive integer. We denote by $B_d: X \times Y \ra \{0,1\}$ 
the Boolean matrix which is defined as follows:
\begin{enumerate}
\item $X = [2^d]$ and $Y = [d]$.
\item 
For every function $b: [d] \ra \{0,1\}$, there exists an $x \in [2^d]$ 
such that, for every $y \in [d]$, we have $B_d(x,y) = b(y)$.  
\end{enumerate}
Note that $B_d$ is unique modulo renaming rows and columns.

\begin{definition}[Shattered sets] \label{def:shattered-sets}
Let $A: X \times Y \ra Z$, with $Z \subseteq \mathbb{R}$, be a matrix and let $J \seq Y$ be a subset of its columns. 
\begin{enumerate}
\item
Suppose that $Z = \{0,1\}$.
We say that $J$ is \emph{VC-shattered} by $A$ if, 
for every function $b: J \ra \{0,1\}$, there exists an $x \in X$ such that, 
for every $y \in J$, we have $B(x,y) = b(y)$.
\item 
We say that $J$ is \emph{P-shattered} by $A$ if
there exists a function $\vec{t}: J \ra \reals$ such that the following holds:
for every function $b: J \ra \{0,1\}$, there exists an $x \in X$ such that, 
for every $y \in J$, we have $A(x,y) \ge \vec{t}(y)$ iff $b(y)=1$. 
\item
Let $\gamma>0$. We say that $J$ is \emph{$P_\gamma$-shattered} by $A$ 
if there exists a function $\vec{t}: J \ra \reals$ such that the following holds:
for every function $b: J \ra \{0,1\}$, there exists an $x \in X$ such that, 
for every $y \in J$, we have 
\[
A(x,y) \left\{ \begin{array}{ll}
         \ge \vec{t}(y)+\gamma & \mbox{if $b(y)=1$} \\
         <   \vec{t}(y)-\gamma & \mbox{if $b(y)=0$}
       \end{array} \right . \enspace .
\] 
\item
We say that $J$ is \emph{V-shattered} by $A$ if there exists a number $t\in \reals$ 
such that the following holds: for every function $b: J \ra \{0,1\}$, there exists 
an $x \in X$ such that, for every $y \in J$, we have $A(x,y) \ge t$ iff $b(y)=1$. 
\item
Let $\gamma>0$. We say that $J$ is \emph{$V_\gamma$-shattered} by $A$ 
if there exists a number $t \in \reals$ such that the following holds:
for every function $b: J \ra \{0,1\}$, there exists an $x \in X$ such that, 
for every $y \in J$, we have 
\[
A(x,y) \left\{ \begin{array}{ll}
         \ge t+\gamma & \mbox{if $b(y)=1$} \\
         <   t-\gamma & \mbox{if $b(y)=0$}
       \end{array} \right . \enspace .
\] 
\end{enumerate}
We will refer to $\vec{t}: J \ra \reals$ occuring in the definition of 
$P$- and the $P_\gamma$-shattered sets as the \emph{thresholds used for 
shattering $J$}. Similarly, we will refer to $t \in \reals$ occuring 
in the definition of $V$- and the $V_\gamma$-shattered sets as the 
\emph{uniform threshold used for shattering~$J$}.
\end{definition}

\begin{definition}[Combinatorial dimensions]
Let $A: X \times Y \ra Z$ be a matrix.
Let $\tau \in \{\mathrm{VC},P,P_\gamma,V,V_\gamma\}$ be one of the
shattering types mentioned in Definition~\ref{def:shattered-sets}.
The \emph{(primal) $\tau$-dimension} of $A$ is the size of a 
largest set $J \seq Y$ that is $\tau$-shattered by $A$ (resp.~$\infty$
if there exist $\tau$-shatterable sets of unbounded size). 
The \emph{dual $\tau$-dimension} of $A$ is defined as the $\tau$-dimension 
of $A^\top$. 
\end{definition}
We use the notations $\VCD(A)$, $\PD(A)$, $\Pfat{\gamma}(A)$, $\VD(A)$
and $\Vfat{\gamma}(A)$ for the (primal) dimensions 
of type $\tau = \mathrm{VC},P,P_\gamma,V,V_\gamma$, respectively. Here, $\VCD(A)$ is the VC-dimension \cite{vapnik1971uniform}, $\PD(A)$ the pseudo-dimension \cite{Pollard1990,Haussler1992}, $\Pfat{\gamma}(A)$ the fat-shattering dimension \cite{Kearns1994}, $\VD(A)$ the Vapnik-dimension \cite{Vapnik1989} and $\Vfat{\gamma}(A)$ the fat-shattered version of the Vapnik-dimension, see, e.g., \cite{Alon1997}.
The corresponding dual dimensions are denoted
by  $\VCD^*(A)$, $\PD^*(A)$, $P_{\gamma}^*(A)$, $\VD^*(A)$ 
and $V_{\gamma}^*(A)$, respectively. 

The matrix \emph{obtained by thresholding the columns of $A: X \times Y \ra Z$
at $\vec{t}: Y \ra \reals$} is defined 
as the Boolean matrix $B: X \times Y \ra \{0,1\}$ such that, for all $x \in X$
and $y \in Y$, we have $B(x,y) = 1$ iff $A(x,y) \ge \vec{t}(y)$.
For $I \seq X$ and $J \seq Y$, we denote the restriction of $A$ to $I \times J$
by $A_{I,J}$. In other words: $A_{I,J}$ is the submatrix of $A$ whose rows
are indexed by $I$ and whose columns are indexed by $J$.
A \emph{witness for the inequality $\PD(A) \ge d$} is defined
as a tripel $(I,J,\vec{t})$ such that the following holds:
\begin{enumerate}
\item 
$I$ is a subset of $X$ of size $2^d$, $J$ is a subset of $Y$ of size $d$
and $\vec{t}: J \ra \reals$.
\item 
Every pattern $b: J \ra \{0,1\}$ occurs in exactly one row of the Boolean matrix
obtained by thresholding the columns of $A_{I,J}$ at $\vec{t}$, i.e., $A_{I,J}$ equals $B_d$ up to permutation of its rows.
\end{enumerate}

\begin{remark} \label{rem:multi-class labels}
Let $k \ge 1$ be a positive integer. 
Consider a matrix $A: X \times Y \ra [k]_0$. It is easy to see that,
if a set $J \seq Y$ can be $P$-shattered by $A$ with 
thresholds $\vec{t}:J \ra \reals$, then it can also be $P$-shattered 
with (suitably chosen) thresholds $\vec{t}:J \ra [k]$. 
An analogous remark applies to $V$-shattering with a uniform threshold $t$.
\end{remark}
When analyzing the $P$- or the $V$-dimension of a matrix with entries 
in $[k]_0$, we will assume that thresholds are taken from $[k]$ whenever 
we find that convenient.

\section{Known relations}
\label{sec:known}
In this section we review some known relations between the 
combinatorial dimensions defined in Section \ref{sec:definitions}.

\subsection{Bounding P- in terms of V-dimension}
It follows directly from the definitions that
\begin{equation}
\VD(A) \le \PD(A)\ \mbox{ and } \Vfat{\gamma}(A) \le \Pfat{\gamma}(A).
\end{equation}
This raises the question whether we can bound the $P$- in terms
of the $V$-dimension (resp.~the $P_\gamma$ in terms of the $V_\gamma$-dimension).
The gap between $\PD(A)$ and $\VD(A)$ cannot be bounded in general, 
as the following well-known example shows.
\begin{example} \label{ex:monotone-functions}
Let $X$ be the set of all monotone\footnote{A function $f : [0,1] \rightarrow [0,1]$ is monotone if $f(x) \leq f(y)$ for all $x \leq y$.} functions from $[0,1]$ to $[0,1]$, $Y = [0,1]$
and $A(x,y) = x(y)$ for $x \in X$. Then $\VD(A) = 1$ and $\PD(A) = \infty$.
\end{example}
In order to bound the $P$- in terms of the $V$-dimension, 
the focus will therefore be on matrices of the form $A: X \times Y \ra [k]_0$.
According to the following results of Ben-David et al.~\cite{BenDavid1995} (here expressed 
in our notation), the $P$- can exceed the $V$-dimension by factor $k$, 
but not by a larger factor\footnote{See \cite[Theorem 7-8]{BenDavid1995} and the 
proof of \cite[Theorem 7]{BenDavid1995}.}:

\begin{theorem}[\cite{BenDavid1995}] \label{th:pdim-ub}
For every matrix $A: X \times Y \ra [k]_0$,
we have 
\begin{equation} \label{eq:pdim-ub}
\PD(A) \le k \cdot \VD(A). 
\end{equation}
\end{theorem}
\begin{theorem}[\cite{BenDavid1995}] \label{th:pdim-lb}
For every $d\ge1$ and every $k\ge1$, there exists a 
matrix $A: X \times Y \ra [k]_0$ such that 
\[ \VD(A) = d\ \mbox{ and }\ \PD(A) = k \cdot d. \]
\end{theorem}

\smallskip\noindent
Alon et al.~\cite{Alon1997} have bounded $P_\gamma$- in terms of the $V_{\gamma/2}$-dimension.

\begin{theorem}[\cite{Alon1997}]  \label{th:fat-ub}
For every matrix $A: X \times Y \ra [0,1]$ and every $0 < \gamma \le 1/2$,
we have\footnote{In~\cite{BenDavid1995}, one finds a factor $2 \lceil 1/(2\gamma) \rceil$
at the place of factor $\lceil 1/\gamma \rceil$. We find the latter (and slightly
smaller) factor preferable because of its simpler form.}
\begin{equation} \label{eq:fat-ub}
\Pfat{\gamma}(A) \le
\left(\left\lceil\frac{1}{\gamma}\right\rceil - 1\right) \cdot \Vfat{\gamma/2}(A) \le
\left(\left\lceil\frac{1}{\gamma}\right\rceil - 1\right) \cdot  \PD(A)
\enspace .
\end{equation}
\end{theorem}

\begin{proof}
The thresholds $t_1,\ldots,t_d$ used for $P_\gamma$-shattering $d := \Pfat{\gamma}(A)$ 
many columns of $A$ must belong to the interval $[\gamma,1-\gamma]$. Any  
threshold $t_i$ can be rounded to the closest multiple of $\gamma$.
Denote the latter by $\hat t_i$. The inequality~(\ref{eq:fat-ub}) becomes now
evident from the following observations. First, by using the thresholds $\hat t_i$ 
instead of $t_i$, the width of shattering may drop from $\gamma$ to $\gamma/2$
(but not beyond). Second, $\hat t_1,\ldots,\hat t_d$ can take on 
at most 
\[ 
r := \left\lceil \frac{1-2\gamma}{\gamma} \right\rceil + 1 = 
\left\lceil\frac{1}{\gamma}\right\rceil - 1\
\]
different values. By the pidgeon-hole principle, there is 
some $t \in \{\hat t_1,\ldots,\hat t_d\}$ that can be used 
for $V_{\gamma/2}$-shattering $d/r$ many points. 
\end{proof}

\subsection{Bounding dual dimension in terms of its primal}
A well-known result due to Assouad \cite{Assouad1983} already mentioned in Section \ref{sec:intro}, which we will refer to as \emph{Assouad's bound}, states that one can upper bound the dual VC-dimension in terms of the (primal) VC-dimension.

\begin{theorem}[\cite{Assouad1983}] \label{th:assouad}
For every matrix $A: X \times Y \ra \{0,1\}$, we have
\begin{equation} \label{eq:assouad}
\VCD^*(A) \le 2^{\VCD(A)+1}-1.
\end{equation}
\end{theorem}
\noindent Note that, under the assumption that $\VCD^*(A)$ is a power of two,
this means
\begin{equation}
\VCD^*(A) \le 2^{\VCD(A)}.
\label{eq:assouad_power}
\end{equation}
The bound in \eqref{eq:assouad_power} is known to be tight for every value of $\VC{A}$, see, e.g., \cite{Matouvsek2002}.

In Appendix \ref{app:assouad} we show that the Assouad's  bound also holds for $\VD(A)$ and $V_{\gamma}(A)$, based on the notion of \emph{uniform $\Psi$-dimension} as defined in \cite{Alon1997}. These observations are summarized in the following result.

\begin{corollary}[Folklore] \label{cor:assouad}
For every matrix $A: X \times Y \ra [0,1]$, we have
\begin{equation} \label{eq:assouad-general}
\VD^*(A) \le 2^{\VD(A)+1}-1\ \mbox{ and }\ 
V_{\gamma}^*(A) \le 2^{\Vfat{\gamma}(A)+1}-1.
\end{equation}
If $\VD^*(A)$, respectively $V_{\gamma}^*(A)$, is a power of two, this means
\begin{equation} \label{eq:assouad-general-power}
\VD^*(A) \le 2^{\VD(A)}\ \mbox{ and }\ 
V_{\gamma}^*(A) \le 2^{\Vfat{\gamma}(A)} .
\end{equation}
\end{corollary}

Combining Theorem~\ref{th:pdim-ub} (applied to $A^\top$)
with Corollary~\ref{cor:assouad}, we directly obtain the following result:

\begin{theorem}[Folklore] \label{cor:dualpdim-ub}
For every matrix $A: X \times Y \ra [k]_0$, the following holds:
\begin{enumerate}
\item
$\PD^*(A) \le k \cdot \left(2^{\VD(A)+1}-1\right) \le k \cdot \left(2^{\PD(A)+1}-1\right)$.
\item
If $\VD^*(A)$ is a power of two, then $\PD^*(A) \le k \cdot 2^{\VD(A)} \leq k \cdot 2^{\PD(A)}$.
\end{enumerate}
\end{theorem}

\noindent Similarly, combining Theorem~\ref{th:fat-ub} 
with Corollary~\ref{cor:assouad}, we directly obtain the following result.   \newpage

\begin{corollary}[Folklore] \label{cor:dualfat-ub}
For every matrix $A: X \times Y \ra [0,1]$, the following holds:
$$
P_{\gamma}^*(A) \le 
\left(\left\lceil\frac{1}{\gamma}\right\rceil-1\right) \cdot
\left(2^{\Vfat{\gamma/2}(A)+1} - 1\right) 
\le 
\left(\left\lceil\frac{1}{\gamma}\right\rceil-1\right) \cdot
\left(2^{\PD(A)+1} - 1\right).
$$
\end{corollary}

\section{Our results} \label{sec:results}
In this section we describe our new contributions, that complement those mentioned in Section \ref{sec:known}. We first discuss results related to the pseudo-dimension.
We start with a result showing that the upper bound on $\PD^*(A)$ 
in Theorem~\ref{cor:dualpdim-ub} can be improved by a factor $2$ 
(roughly) for matrices $A$ with $\VD(A) = 1$. 

\begin{theorem} \label{th:dualpdim-ub}
Let $A: X \times Y \ra [k]_0$ with $k\ge1$ be a matrix with $\VD(A) = 1$. Then $\PD^*(A) \le k+2$.
\end{theorem}

\smallskip
The next result implies that the upper bound on $\PD^*(A)$ in the second statement of
Theorem~\ref{cor:dualpdim-ub} is tight for matrices with $\VD(A) \geq 2$, as well as the upper bound on $\PD^*(A)$ in Theorem~\ref{th:dualpdim-ub} whenever $\VD(A) \geq 1$.

\begin{theorem} \label{th:dualpdim-lb} The following two lower bounds hold:
\begin{enumerate}
\item
For every $d\ge2$ and every $k\ge1$, there exists a matrix $A: X \times Y \ra [k]_0$
such that 
\[ 
\PD(A) = d\ ,\ \VD^*(A) = 2^d\ 
\mbox{ and }\ \PD^*(A) = k \cdot 2^d. 
\]
\item
For every $k\ge1$, there exists a matrix $A: X \times Y \ra [k]_0$
such that $\VD(A) = \PD(A) = 1$ and $\PD^*(A) = k+2$.
\end{enumerate}
\end{theorem}
In combination with the a technical tool defined in Section \ref{subsec:tools}, we also obtain the following corollary. It stands in stark contrast to Assouad's bound for the VC-dimension.

\begin{corollary}
There exist a matrix $A : X \times Y \rightarrow [0,1]$, such that $\PD(A) = 1$ and $\PD^*(A) = \infty$.
\label{cor:infty}
\end{corollary}

We next move to our results for the fat-shattering dimensions.
The first result here implies that upper bound on $\Pfat{\gamma}(A)$ from 
Theorem~\ref{th:fat-ub} is tight up to a small constant factor:
 
\begin{theorem} \label{th:fat-lb}
For every $d\ge1$, there exists a matrix $A: X \times Y \ra [0,1]$
such that $\PD(A) = d$ and, for all $k\ge1$, 
\[ \Pfat{1/(2k)}(A) \ge k \cdot d \enspace . \]
\end{theorem}

Finally, our last results state that the bound on $P_{\gamma}^*(A)$ from 
Corollary~\ref{cor:dualfat-ub} is tight up to a small constant factor.

\begin{theorem} \label{th:dualfat-lb}
The following two lower bounds hold:
\begin{enumerate}
\item
For every $d \ge 2$, there exists a matrix $A: X \times Y \ra [0,1]$
such that $\PD(A) = d$ and, for all $k\ge1$,
\[
P_{1/(2k)}^*(A) = k \cdot 2^d \enspace . 
\]
\item
There exists a matrix $A: X \times Y \ra [0,1]$ 
such that $\PD(A)=1$ and, for all $k\ge2$,
\[ P_{1/(2k)}^*(A) = k+2 \enspace . \]
\end{enumerate}
\end{theorem}

\section{Proofs}
\label{sec:proofs}
Section~\ref{subsec:dualpdim-ub} is devoted to the proof of 
Theorem~\ref{th:dualpdim-ub}. In Section~\ref{subsec:tools},
we make some considerations which will allow for an easier
presentation of our lower bound constructions, that are given 
in Sections~\ref{subsec:dualpdim-lb} and~\ref{subsec:fat-lb}.

\subsection{Proof of Theorem~\ref{th:dualpdim-ub}} \label{subsec:dualpdim-ub}

For $k=1$, the assertion of the theorem collapses to the claim that $\VCD^*(A) \le 3$
for every Boolean matrix $A$ with $\VD(A) = 1$. This is an immediate
consequence of~(\ref{eq:assouad}). Suppose now that $k \ge 2$. It suffices 
to show that $\PD^*(A) \ge k+3$ implies that $\VD(A) \ge 2$ (i.e., we give a proof by contradiction). 
Pick a witness $(I,J,\vec{t})$ for $\PD^*(A) \ge k+3$. More concretely:
\begin{itemize}
\item
$I = \{x_1,\ldots,x_{k+3}\}$,  $J \seq Y$ with $|J|=2^{k+3}$
and $\vec{t}:I \ra [k]$, say $\vec{t}(x_i) = t_i$.
\item
The matrix obtained by thresholding the rows of $A_{I,J}$ at $\vec{t}$ 
equals $B_{k+3}^\top$.
\end{itemize}
We may assume that, after renumbering the rows appropriately,
one has $t_1 \le\ldots\le t_{k+3}$. We decompose the rows of $A_{I,J}$ into
maximal blocks such that the same threshold is assigned to every row from
the same block. Since any threshold $t_i$ is taken from $[k]$, the total
number $k'$ of blocks is bounded by $k$. A block that is different from the
first and from the last block is said to be an \emph{inner block}. We proceed
by case analysis:
\begin{description}
\item[Case 1:] One of the blocks contains 4 rows. \\
Then $\VD^*(A) \ge 4$ is immediate.
Thus $\VD(A) \ge \lfloor \log \VD^*(A) \rfloor \ge 2$.
\item[Case 2:] The first or the last block contains 3 rows. \\
For reasons of symmetry, we may assume that the first block contains 3 rows.
Consider the following $(4 \times 2)$-submatrix of $B_{k+3}^\top$:
\[
\begin{array}{cc}
  0 & 0 \\
  0 & 1 \\
  1 & 0 \\
  \hline
  1 & 1
\end{array}
\]
The first three rows are taken from the first block and the last row is taken
from the last block. The separation line between the third and the last row
is only intended to illustrate the transition from one block to another.
Remember that the rows of the first block of $A_{I,J}$ are thresholded
at $t_1$ while the rows of the last block are thresholded at $t_{k'} > t_1$.
Hence, if we threshold \emph{all rows} (or \emph{all columns}) of $A_{I,J}$
at $t_1$, then the above submatrix of $B_{k+3}^\top$ will remain unchanged.
We may conclude from this discussion that $\VD(A) \ge 2$.

\item[Case 3:] One of the inner blocks contains 2 rows, say block $b$. \\
The argument is similar to that given in Case 2. The relevant submatrix of $B_{k+3}^\top$
(with one row of the first block, two rows of block $b$, one row of the last
block and two separation lines inbetween) now looks as follows:
\[
\begin{array}{cc}
  0 & 0 \\
  \hline
  0 & 1 \\
  1 & 0 \\
  \hline
  1 & 1
\end{array}
\]
Since $t_1 < t_b < t_{k'}$, thresholding all rows (or all columns) of $A_{I,J}$
at $t_b$ will leave the above submatrix of $B_{k+3}^\top$ unchanged.
We may conclude that $\VD(A) \ge 2$.
\end{description}
Since $A_{I,J}$ has $k+3$ rows (with $k\ge2$), it is easy to argue 
that one of the three above cases must occur. Suppose first that $k=2$. 
Then there at most $2$ blocks and $5$ rows. It follows that the first 
or the last block contains at least $3$ rows. 
Suppose now that $k\ge3$. If the first and the last block contain at most 
two rows, respectively, then at least $k-1$ rows are left 
for the $k'-2 \le k-2$ inner blocks.  By the pidgeon-hole principle, 
there must be an inner block with two rows. This completes the proof of Theorem~\ref{th:dualpdim-ub}.

\subsection{Preliminaries for lower bound constructions} \label{subsec:tools}

Consider again the Boolean matrix $B_d$ with $d$ columns and $2^d$ rows
that had been defined in Section~\ref{sec:definitions}. It is evident 
that $B_d$ satisfies the following conditions:

\begin{enumerate}[i)]
\item \textbf{Distinctness Condition:} The rows of $B_d$ are pairwise distinct.
\item \textbf{General Balance Condition:}
For any $k \in [d]$, any choice of $k$ distinct columns of $B_d$
and any pattern $\vec{b} \in \{0,1\}^k$, there are exactly $2^{d-k}$ 
rows of $B_d$ which realize the pattern $\vec{b}$ within the chosen 
columns.
\setcounter{balance}{\value{enumi}}
\end{enumerate}
The general balance condition implies the following:
\begin{enumerate}[i)]
\setcounter{enumi}{\value{balance}}
\item \textbf{1st Balance Condition:}
Each column of $B_d$ has as many zeros as ones.
\item \textbf{2nd Balance Condition:}
For any two distinct columns of $B_d$, any pattern from $\{0,1\}^2$ is 
realized within these columns by the same number of rows.
\end{enumerate}
\begin{remark}[Proof templates]
\noindent Consider a matrix $A: X \times Y \ra [k]_0$.
The following template for proving assertions like $\PD(A) \le d$
will prove itself quite useful.
\begin{itemize}
\item Assume for contradiction that $\PD(A) \ge d+1$.
\item Pick a witness $(I,J,\vec{t})$ for this inequality.
\item 
Exploit the fact that the matrix $B$ obtained by thresholding 
the columns of $A_{I,J}$ at $\vec{t}$ must be equal to $B_{d+1}$.
\item
Prove that $B$ violates one of the conditions that $B_{d+1}$
must satisfy.
\end{itemize}
Sometimes the following (slightly simpler) template can be used instead:
\begin{itemize}
\item Take a fixed but arbitrary function $\vec{t}: Y \ra [k]$.
\item Let $B$ be the matrix obtained by thresholding the columns of $A$ at $\vec{t}$.
\item 
Show that no more than $d$ columns of $B$ have at least $2^d$ zeros 
and at least $2^d$ ones.
\end{itemize}
This also shows that $\PD(A) \le d$ because no submatrix of $B$
with $d+1$ columns and $2^{d+1}$ rows has a chance to satisfy
the first balance condition.
\end{remark}
We next introduce matrices that, though not being Boolean, are 
close relatives of the matrix $B_d$.
\begin{definition}
Let $k$ and $d_1,\ldots,d_k$ be positive integers and let $D = d_1 +\ldots+ d_k$
denote their sum. The \emph{$B_D$-based matrix with $k$ column blocks of 
sizes $d_1,\ldots,d_k$} is the matrix $A: X \times Y \ra [k]_0$, 
where $X = [2^D]$ and $Y = [D]$, that results from the following procedure:
\begin{enumerate}
\item
Decompose the $D$ columns of $B_D$ into $k$ blocks of sizes $d_1,\ldots,d_k$.
The blocks are consecutively numbered from $1$ to $k$.
\item
Obtain $A$ from $B_D$ by replacing any $1$-entry (resp.~$0$-entry) in a column
belonging to block $b \in [k]$ by $b$ (resp.~by $b-1$). 
\end{enumerate}
The \emph{$B_D^\top$-based matrix with $k$ row blocks of sizes $d_1,\ldots,d_k$} is defined analogously. 

\end{definition}
Note that the matrix $A$ resulting from the above procedure has the property
that, for any two columns $y_1$ in block $b_1$ and $y_2$ in block $b_2>b_1$
and any row $x$, we have $A(x,y_1) \le A(x,y_2)$. We will refer to this
property as \emph{block monotonicity}. 

At this point we also bring into play the matrix $\dot{A}$, which is defined as the matrix $A$ augmented with a row of zeros.  
Formally, we assume that $0 \notin X$ 
and define $\dot{A}: (X \cup \{0\}) \times Y \rightarrow Z$ as the extension of $A$ 
which satisfies $\dot{A}(0,y)=0$ for all $y \in Y$.  The (technical) use of $\dot{A}$ will become clear in Section \ref{subsec:fat-lb} (in particular, this is explained after Definition \ref{def:merge}), but it is already included in the statements that follow.

\begin{lemma} \label{lem:shatter-matrix-primal}
Let $D = d_1 +\ldots+ d_k$ and let $A$ be the $B_D$-based matrix 
with $k$ column blocks of sizes $d_1,\ldots,d_k$. 
Then $\VD(A) = \VD(\dot{A}) = \max_{j\in[k]}d_j$ and $\PD(A) = \sum_{j=1}^{k}d_j$.
\end{lemma}

\begin{proof}
We first show that the pseudo-dimension of $A$ equals $D$. 
Let $\vec{t}: [D] \ra [k]$ be the mapping that assigns to every column 
in block $j\in[k]$ the threshold $j$. Then the matrix obtained 
by thresholding the columns of $A$ at $\vec{t}$ equals $B_D$. 
It follows that $\PD(A) \ge \VCD(B_D) = D$. 
Of course $\PD(A)$ cannot exceed $D$ so that $\PD(A) = D$.

Next, set $d_{max} = \max_{j \in [k]}d_j$. Pick some index $j_{max}\in[k]$
such that $d_{j_{max}} = d_{max}$.
We still have to show that $\VD(A) = \VD(\dot{A}) = d_{max}$. 
Thresholding the columns of $A$ at the uniform threshold $j_{max}$,
we obtain a matrix $B$ that equals $B_D$ within block $j_{max}$.
This shows that $\VD(A) \ge d_{max}$. The inequality $\VD(\dot{A}) \le d_{max}$
can be seen as follows. Pick a fixed but arbitrary $J \seq [D]$ 
of size $1+d_{max}$ and a fixed but arbitrary uniform threshold $t \in [k]$. 
Let $B$ be the matrix obtained by thresholding the columns of $\dot{A}$
at $t$. The set $J$ must contain two columns belonging to two different blocks, 
say column $y_1$ in block $b_1$ and column $y_2$ in block $b_2 > b_1$. 
By the block-monotonicity of $A$ (which implies block-monotonicity for $\dot{A}$ as well), no row of $B$ can assign label $1$
to $y_1$ and label $0$ to $y_2$. Since $J$ and $t$ were arbitrary choices,
it follows that no set of size $1+d_{max}$ can be $V$-shattered by $\dot{A}$.
\end{proof}

Setting $d_1 =\ldots= d_k = d$ in Lemma~\ref{lem:shatter-matrix-primal},
we obtain the following result (which is almost the same as 
Theorem~\ref{th:pdim-ub}):

\begin{corollary} \label{cor:pdim-lb}
For every $d\ge1$ and every $k\ge1$, there exists a matrix $A: X \times Y \ra [k]_0$
such that $\VD(A) = \VD(\dot{A}) = d$ and $\PD(A) = k \cdot d$.
\end{corollary}

\subsection{Proof of Theorem~\ref{th:dualpdim-lb}} \label{subsec:dualpdim-lb}

Theorem~\ref{th:dualpdim-lb} is a direct consequence of the following 
two results:

\begin{lemma} \label{lem1:shatter-matrix-dual}
Let $d\ge2$ and $k\ge 1$ be given. For $D = k\cdot 2^d$, let $A$ be the $B_D^\top$-based 
matrix with $k$ row blocks of size $2^d$, respectively. Then 
\[ 
\PD(A) = \PD(\dot{A}) = d\ ,\ \VD^*(A) = 2^d\ \mbox{ and }\ 
\PD^*(A) = k \cdot 2^d \enspace .
\]
\end{lemma}

\begin{proof}
The identities $\VD^*(A) = 2^d$ and $\PD^*(A) = k \cdot 2^d$
are immediate from Lemma~\ref{lem:shatter-matrix-primal}.
Hence it suffices to verify the identity $\PD(A) = \PD(\dot{A}) = d$.
We can infer from $\VD^*(A) = 2^d$ 
that $\PD(\dot{A}) \ge \PD(A) \ge \VD(A) \ge d$. Hence the proof
can be accomplished by showing that $\PD(\dot{A}) \le d$.
For sake of brevity, set 
\[ s = 2^d\ \mbox{ and }\ \bar{d} = 1+d \enspace . \]
Assume for contradiction that $\PD(\dot{A}) \ge \bar{d}$ 
and fix some witness $(I,J,\vec{t})$ for this inequality, i.e.,
\begin{enumerate}
\item 
$I \seq [D]_0$, $|I| = 2^{\bar d}$, $J \seq [2^{D}]$, $|J|= \bar d$
and $\vec{t}: J \ra [k]$ assigns a threshold to each column 
of $A_{I,J}$.\footnote{Recall from the definition of $\dot{A}$ that
this matrix is obtained from $A$ by adding an all-zeros row which is
indexed by $0$.}
\item
The matrix $B$ obtained by thresholding the columns of $A_{I,J}$ at $\vec{t}$
equals $B_{\bar d}$. 
\end{enumerate}
Before we proceed with the proof, we fix some notation.
For $b=1,\ldots,k$, let $I_b$ denote the set of row indices in $I$ that
belong to block $b$ of $A$. Set $I_0 = I \cap \{0\}$ and note that
\[
\sum_{b=0}^{k}|I_b| = |I| = 2^{\ol{d}} = 2s\ \mbox{ (twice the block size)}
\]
while, for every $b \in [k]$, we clearly have $0 \le |I_b| \le s$.
Let $b_0,b'_0$ (resp.~$b_1,b'_1$) denote the smallest and second-smallest
(resp.~largest and second-largest) $b \in [k]$ such
that $|I_b| \neq 0$. \\
An obvious question is whether $0 \in I$, that is, whether the extra all-zeros
row is among the rows of $B$. We claim that this is not the case.
Assume for contradiction that $0 \in I$. We proceed by case analysis:
\begin{description}
\item[Case 1:] $1 \le |I_{b_0}| \le s-1$. \\
Pick an arbitrary but fixed column $j \in J$ of $B$. In order to
satisfy the first balance condition, the threshold $t_j$ must be large enough
so that in block $b_0$ of column $j$ only 0-entries are found.\footnote{A single
1-entry in this block would imply that we have only 1-entries in all subsequent blocks.}
It follows that any row of $B$ belonging to block $b_0$ has 0-entries only
and therefore coincides with the extra all-zeros row.
This is in contradiction with the distinctness condition.
\item[Case 2:] $|I_{b_0}| = s$. \\
Pick an arbitrary but fixed column $j \in J$ of $B$. In order to
satisfy the first balance condition, the threshold $t_j$ must be large enough
so that in block $b_0$ of column $j$ at least $s-1$ 0-entries are found.
Pick another column $j' \neq j$ in $B$ (also with at least $s-1$
0-entries in block $b_0$). Then the pattern $00$ occurs in columns $j$
and $j'$ of $B$ at least $s-1$ times (one time in row $0$ and at
least $s-2$ times in block $b_0$). But then $B$ must have at
least  $4(s-1)$ rows in order to satisfy the second balance condition.
Hence $4(s-1) \le 2s$ (because $B$ has $2s$ rows).
It follows that $s \le 2$, which is in contradiction
with our assumptions that $d \ge 2$ and $s = 2^d \ge 4$.
\end{description}
In any case, we arrived at a contradiction, which proves the above claim
that $0 \notin I$. In order to accomplish the proof, we still have to derive a final contradiction.
We proceed by case distinction again.
\begin{description}
\item[Case A:] $|I_{b_0}| \ge 2$. \\
In order to satisfy the first balance condition, the threshold $t_j$ of any
column $j \in J$ must be large enough so that in block $b_0$ of this column
only 0-entries are found. Thus all rows of $B$  belonging to block $b_0$
realize the all-zeros pattern, which is in contradiction with the distinctness
condition.
\item[Case B:] $|I_{b_0}|=1$ and $|I_{b'_0}| \le s-1$. \\
The argument is similar. Now the single row of $B$ belonging to block $b_0$
and all rows of $B$ belonging to block $b'_0$ realize the all-zeros pattern.
\item[Case C:] $|I_{b_1}| \ge 2$ or $(|I_{b_1}|=1$ and $|I_{b'_1}| \le s-1$). \\
Then, for reasons of symmetry, the last two rows of $B$ both realize
the all-ones pattern.
\item[Case D:] $|I_{b_0}| = |I_{b_1}| = 1$ and $|I_{b'_0}| = |I_{b'_1}| = s$. \\
Since $|I| = 2s$, this case can occur only if $b'_0 = b'_1$.
But then $2s = |I| = |I_{b_0}| + |I_{b'_0}| + |I_{b_1}| = s+2$
so that $s=2$. This contradicts to our assumption $d\ge2$ and $s=2^d\ge4$.
\end{description}
In any case, we arrived at a contradiction.
\end{proof}

\begin{lemma} \label{lem2:shatter-matrix-dual}
Let $k \ge 2$ and let$A$ be the $B_{k+2}^\top$-based matrix 
with $k$ row blocks of sizes 
$$
d_j = \left\{\begin{array}{ll}
2 & \text{ for } j = 1,k\\
1 & \text { for } j = 2,\dots,k-1
\end{array}\right..
$$
Then $\PD(A) = \PD(\dot{A}) = 1$ and $\PD^*(A) = k+2$.
\end{lemma}

\begin{proof}
The identity $\PD^*(A) = k+2$ is immediate from 
Lemma~\ref{lem:shatter-matrix-primal}. Clearly $\PD(\dot{A}) \ge \PD(A) \ge 1$.
Hence it suffices to show that $\PD(\dot{A}) \le 1$. The rows of $\dot{A}$ 
have indices $0,1,\ldots,k+2$ and index $0$ is reserved for the all-zeros row. 
Assume for contradiction that $\PD(\dot{A}) \ge 2$ and pick a witness $(I,J,t)$ 
for this inequality so that the following holds:
\begin{itemize}
\item
$J = \{j_1,j_2\} \subset [2^{k+2}]$, $I \seq [k+2]_0$ with $|I|=4$
and $\vec{t}: J \ra [k]$, say $\vec{t}(j_1) = t_1$ and $\vec{t}(j_2) = t_2$.
\item
The matrix $B$ obtained by thresholding the columns of $\dot{A}_{I,J}$ 
at $\vec{t}$ equals $B_2$ (with rows indexed by $I$ and columns indexed by $J$).
\end{itemize}
Consequently $B$ satisfies the distinctness condition and the balance
conditions. Consider the smallest index $i_1$ and the second-smallest 
index $i_2$ in $I$. Note that, since $|I|=4$ and the last block of $B$ 
is of size $2$, neither $i_1$ nor $i_2$ belongs to the last block, i.e.,
either $i_1,i_2 \in \{0,1,2\}$ or $i_2$ belongs to one of the inner blocks
consisting of a single row only. In order to establish the first balance condition
for the matrix $B$, the thresholds $t_1$ and $t_2$ must be large enough so
that only zeros are found in the first two components (indexed by $i_1$ and $i_2$)
of the columns $j_1$ and $j_2$. Thus the first two rows of $B$ both realize the
all-zeros pattern, which is in contradiction with the distinctness condition.
\end{proof}

Lemma~\ref{lem2:shatter-matrix-dual} does not cover the case $k=1$
in the second assertion of Theorem~\ref{th:dualpdim-lb}. But this 
case is easy to handle: setting $A = B_3^\top$, 
we obtain $\PD^*(A) = \VCD^*(A) = 3$ and $\PD(A) = \VCD(A) = 1$.

\subsection{Proofs of Theorems~\ref{th:fat-lb} and~\ref{th:dualfat-lb}}
\label{subsec:fat-lb}

Matrices $A$ with the properties as prescribed by Theorems~\ref{th:fat-lb} 
and~\ref{th:dualfat-lb} are easy to construct by means of a suitable operation
that merges matrices of a given matrix family into a single matrix. 

\begin{definition}[Merge-operation]
Let $(A_k)_{k\ge1}$ with $A_k: X_k \times Y_k \ra [k]_0$ be a given family
of matrices. Let $X$ (resp.~$Y$) denote the disjoint union of the sets $X_k$
(resp.~$Y_k$) with $k \ge 1$. Assume that $X \cap Y = \eset$. 
For every $x \in X$, let $k(x)$ denote the unique $k$ such that $x \in X_k$. 
The notation $k(y)$ is understood analogously. 
The matrix $A: X \times Y \ra [0,1]$ given by
\[
A(x,y) = \left\{ \begin{array}{ll}
                        \frac{A_{k(x)}(x,y)}{k(x)} & \mbox{if $k(y)=k(x)$} \\
                        0           & \mbox{otherwise}
                        \end{array} \right. \enspace ,
\]
is called the \emph{merge of the family $(A_k)_{k \ge 1}$}.
\label{def:merge}
\end{definition}

The merge-operation reveals why we introduce the matrix $\dot{A}$: The pseudo-dimension (or any other combinatorial dimension for that matter) of the matrix $A$ restricted to the columns $Y_k$ is nothing more than the pseudo-dimension of the functions in $A_k$ augmented with an infinite number of functions that are zero everywhere. The pseudo-dimension of this function class clearly equals the pseudo-dimension of the matrix $\dot{A}_k$. The merge-operation has the following properties:

\begin{lemma} \label{lem:merge}
Let $A$ be the merge of the family $(A_k)_{k\ge1}$. Then the following holds:
\begin{enumerate}
\item $\Pfat{1/(2k)}(A) \ge \PD(A_k)$.
\item  Let $d_0 \in \mathbb{N}$. If $\sup_k \PD(\dot{A}_k) \leq d_0$, then $\PD(A) \le d_0$.
\item $\Vfat{1/(2k)}(A) \ge \VD(A_k)$.
\item  Let $d_0 \in \mathbb{N}$. If $\sup_k \VD(\dot{A}_k) \leq d_0$, then $\VD(A) \le d_0$.
\end{enumerate}
\end{lemma}

\begin{proof}
We only prove the first two assertions of the lemma; the other two assertions are quite similar. 

Note that, for $k=k(x)=k(y)$, $A$ coincides with $A_k$ except for
scaling down the values $0,1,\ldots,k$ by factor $k$.
Since $A_k$ takes integer values, each set that can be $P$-shattered 
by $A_k$ can actually be $P_{1/2}$-shattered. After down-scaling,
the width of shattering becomes $1/(2k)$. From these observations,
the first assertion of the lemma easily follows. 

We proceed with the proof of the second assertion.
Set $d := \PD(A)$. Fix some witness $(I,J,\vec{t})$ so that the
following holds:
\begin{enumerate}
\item $I \subset X$, $|I| = 2^d$, $J \subset Y$, $|J| = d$
and $\vec{t}: J \ra \nats$ assigns a threshold $t_y := \vec{t}(y)$ 
to every $y \in J$.
\item
The matrix $B$ obtained by thresholding the columns of $A_{I,J}$ 
at $\vec{t}$ equals $B_d$ (with rows indexed by $I$ and columns 
indexed by $J$). 
\end{enumerate}
It follows that $B$ satisfies the distinctness condition and the balance 
conditions.\\

\noindent \textbf{Claim 1:} For every $y \in J$, we have $t_y>0$.
\begin{proof}
$t_y\le 0$ would imply that column $y$ of $B$ has no 0-entry,
which is in contradiction with the first balance condition.
\end{proof}

\noindent \textbf{Claim 2:} The mapping $y \mapsto k(y)$ assigns the same value
to all $y \in J$.
\begin{proof}
Assume to the contrary that there exist $y_1,y_2 \in J$ such that $k(y_1) \neq k(y_2)$.
Then, for every row $x$ of $B$, at least one of the entries $B[x,y_1]$ and $B[x,y_2]$
equals $0$ (because $k(x)$ cannot be equal to both, $k(y_1)$ and $k(y_2)$).
By the first balance condition, any column in $B$ has as many 0- as 1-entries. 
Since this is particularly true for the columns $y_1$ and $y_2$,
it follows that, for every row $x$ of $B$, exactly one of the 
entries $B[x,y_1]$ and $B[x,y_2]$ equals $0$.
Thus column $y_2$ of $B$ is the entry-wise logical negation of the column $y_1$. 
This, however, is in contradiction with the second balance condition.
\end{proof}

\noindent \textbf{Claim 3:} Let $k_1$ denote the common $k$-value of $y_1,\ldots,y_{d}$.
Then any row $x$ in $B$ with $k(x) \neq k_1$ has 0-entries only.
\begin{proof}
This is straighforward.
\end{proof}

We conclude from Claims 2 and 3
that $\PD(A) = d \le \PD(\dot{A}_{k_1})$ and, by assumption,
the latter quantity is at most $d_0$, which concludes the proof.
\end{proof}

Theorem~\ref{th:fat-lb} is now a direct consequence of 
Lemma~\ref{lem:merge} in combination with Corollary~\ref{cor:pdim-lb},
while Theorem~\ref{th:dualfat-lb} is a direct consequence of 
Lemma~\ref{lem:merge} in combination with Lemmas~\ref{lem1:shatter-matrix-dual}
and~\ref{lem2:shatter-matrix-dual}. In order to prove Corollary \ref{cor:infty}, note that Lemma \ref{lem2:shatter-matrix-dual} tells us that for every $k$ there exists a matrix $A_k$ such that $\PD(\dot{A}_k) = 1$ and $\PD^*(\dot{A}_k) \geq \PD^*(A_k) = k+2$. We may then apply Lemma \ref{lem:merge}.

\paragraph{Acknowledgements.} The first author thanks Tim Roughgarden for discussions that (indirectly) lead to the questions studied in this work.

\bibliographystyle{plain}
\bibliography{references}

\newpage
\appendix
\section{On the derivation of Assouad's bound for uniform dimensions}\label{app:assouad}

We say that $J \subseteq Y$ is \emph{VC-shattered} by $A : X \times Y \rightarrow  \{0,1,*\}$ if, 
for every function $b: J \ra \{0,1\}$, there exists an $x \in X$ such that, 
for every $y \in J$, we have $B(x,y) = b(y)$.
We first note that~(\ref{eq:assouad}) is also valid for
every matrix of the form $A: X \times Y \ra \{0,1,*\}$:
the central observation in the proof is that $B_d$ 
contains $B^\top_{\lfloor \log d \rfloor}$ as a submatrix.
This implies that $\VCD(A) \ge \lfloor \log \VCD^*(A) \rfloor$,
which is equivalent to~(\ref{eq:assouad}).

Consider now a matrix of the general form $A: X \times Y \ra Z$ with $Z \seq \reals$.
Making use of the concept of uniform $\Psi$-dimensions from~\cite{BenDavid1995}, 
the result of Assouad can be extended to several other combinatorial dimensions.
Let $\Psi$ denote a family of substitutions of the form $\psi:\reals \ra \{0,1,*\}$.
Denote by $\psi(A)$ the matrix obtained from $A$ by performing the substitution $\psi$ 
entry-wise. The \emph{uniform $\Psi$-dimension} of $A$ is then defined as 
\[ \psidim_U(A) = \sup_{\psi\in\Psi}\VCD(\psi(A)) \enspace . \]
Let $\Psi_Y$ denote the set of all collections $\bar\psi = (\psi_y)_{y \in Y}$ 
with $\psi_y \in \Psi$. Denote by $\bar\psi(A)$ the matrix obtained from $A$ 
by replacing each entry $A(x,y)$ with $\psi_y(A(x,y))$. 
The \emph{(non-uniform) $\Psi$-dimension} of $A$ is defined as
\[ \psidim(A) = \sup_{\bar\psi\in\Psi_Y}\VCD(\bar\psi(A)) \enspace . \]
As usual, we get the corresponding dual dimensions 
by setting $\psidim^*(A) = \psidim(A^\top)$ and $\psidim^*_U(A) = \psidim_U(A^\top)$.
Note that $\psi(A^\top) = \psi(A)^\top$ while $\bar\psi(A^\top)$ is not generally equal
to $\bar\psi(A)^\top$.

As noted in~\cite{BenDavid1995}, several popular combinatorial dimensions 
can be viewed as (uniform or non-uniform) $\psi$-dimension. Here we are
particularly interested in the $P$-, $P_\gamma$, $V$-and $V_\gamma$-dimension:

\begin{remark} We next explain how to interpret known dimensions as special cases of the $\Psi$-dimension.
\begin{enumerate}
\item
If $\Psi$ is the set of mappings $\psi_t$ of the form $\psi_t(a) = \sgn(a-t)$
for some $t \in \reals$, then $\psidim(A) = \PD(A)$ 
and $\psidim_U(A) = \VD(A)$ (see~\cite{BenDavid1995}).
\item
If $\Psi$ is the set of mappings $\psi_t$ of the form 
\[
\psi_t(a) = \left\{ \begin{array}{ll}
              1 & \mbox{if $a \ge t+\gamma$} \\
              0 & \mbox{if $a < t-\gamma$} \\
              * & \mbox{otherwise}
            \end{array} \right.
\]
for some $t \in \reals$ and $\gamma > 0$, then $\psidim(A) = \Pfat{\gamma}(A)$
and $\psidim_U(A) = V_{\gamma}(A)$. 
\end{enumerate}
\end{remark}

The following calculation, with $\psi$ ranging over all functions in $\Psi$, 
shows that Theorem~\ref{th:assouad} can be extended to any uniform $\Psi$-dimension 
at the place of the VC-dimension:
\begin{eqnarray*}
\psidim_U^*(A) & = & \psidim_U(A^\top) = \sup_{\psi}\VCD(\psi(A^\top))
= \sup_{\psi}\VCD(\psi(A)^\top) \\
& = & \sup_{\psi}\VCD^*(\psi(A)) \le \sup_{\psi}\left(2^{\VCD(\psi(A))+1}-1\right) \\
& = & 2^{\sup_{\psi}\VCD(\psi(A))+1}-1 = 2^{\psidim_U(A)+1}-1 \enspace .
\end{eqnarray*}
We remark that a similar argument for the non-uniform $\Psi$-dimension fails as it then no longer holds that $\bar{\psi}(A^\top) = \bar{\psi}(A)^\top$ (which is the argument we use in the third equality above).

\end{document}

%% file: macros.tex
\newtheorem{theorem}{Theorem}[section]
\newtheorem{lemma}[theorem]{Lemma}

\newtheorem{remark}[theorem]{Remark}

\newtheorem{definition}[theorem]{Definition}
\newtheorem{corollary}[theorem]{Corollary}

\newtheorem{example}[theorem]{Example}

\newcommand{\sgn}{\mbox{sgn}}

\newcommand{\seq}{\subseteq}

\newcommand{\ra}{\rightarrow}

\newcommand{\beq}{\begin{equation}}
\newcommand{\eeq}{\end{equation}}
\newcommand{\eset}{\emptyset}
\newcommand{\ol}{\overline}

\newcommand{\nats}{\mathbbm{N}}

\newcommand{\VCD}{\mathrm{VC}}
\newcommand{\VD}{\mathrm{Vdim}}
\newcommand{\PD}{\mathrm{Pdim}}
\newcommand{\psidim}{\mathrm{\Phi}}
\newcommand{\Pfat}[1]{\mathrm{P_{#1}}}
\newcommand{\Vfat}[1]{\mathrm{V_{#1}}}

\newcommand\reals{\mathbbm{R}}